\documentclass[11pt,twoside,letterpaper,reqno,draft]{amsart}

\usepackage{amsmath,amssymb,amscd,latexsym,amsthm,bm,tabularx}
\usepackage{times}

\newtheorem{theorem}{Theorem}
\newtheorem{corollary}{Corollary}
\newtheorem{lemma}{Lemma}

\begin{document}

\title{On two problems of Carlitz and their generalizations}

\author{Ioulia N. Baoulina}

\address{Department of Mathematics, Moscow State Pedagogical University, Krasnoprudnaya str. 14, Moscow 107140, Russia}
\email{jbaulina@mail.ru}

\date{}

\maketitle

\begin{abstract}
Let $N_q$ be the number of solutions to the equation
$$
(a_1^{}x_1^{m_1}+\dots+a_n^{}x_n^{m_n})^k=bx_1^{k_1}\cdots x_n^{k_n}
$$
over the finite field $\mathbb F_q=\mathbb F_{p^s}$. Carlitz found formulas for~$N_q$ when $k_1=\dots=k_n=m_1=\dots=m_n=1$, $k=2$, $n=3$ or $4$, $p>2$; and when ${m_1=\dots=m_n=2}$,\linebreak $k=k_1=\dots=k_n=1$, $n=3$ or $4$, $p>2$. In earlier papers, we studied the\linebreak above equation with $k_1=\dots=k_n=1$ and obtained some generalizations of Car-\linebreak litz's results. Recently, Pan, Zhao and Cao considered the case of arbitrary positive\linebreak integers $k_1,\dots,k_n$ and proved the formula $N_q=q^{n-1}+(-1)^{n-1}$, provided that\linebreak $\gcd(\sum_{j=1}^n (k_jm_1\cdots m_n/m_j)-km_1\cdots m_n,q-1)=1$. In this chapter, we determine $N_q$ explicitly in some other cases.
\end{abstract}



\thispagestyle{empty}

\section{Introduction}

Let $\mathbb F_q$ be a finite field of characteristic~$p$ with $q=p^s$ elements, $\mathbb F_q^*=\mathbb F_q^{}\setminus\{ 0\}$. For $p>2$, let $\eta$ denote the quadratic character on $\mathbb F_q$ ($\eta(x)=+1,-1,0$ according as $x$ is a square, a nonsquare or zero in~$\mathbb F_q$). We consider an equation of the type
\begin{equation}
\label{eq1}
(a_1^{}x_1^{m_1}+\dots+a_n^{}x_n^{m_n})^k=bx_1^{k_1}\cdots x_n^{k_n},
\end{equation}
where $a_1,\dots,a_n,b\in\mathbb F_q^*$, $k,k_1,\dots,k_n, m_1,\dots,m_n$ are positive integers, and $n\geqslant 2$. By $N_q$ denote the number of solutions $(x_1,\dots,x_n)\in\mathbb F_q^n$ to (\ref{eq1}). Let
\begin{align*}
k_0&=\gcd(k,k_1,\dots,k_n,q-1),\qquad M=\text{lcm}[m_1,\dots,m_n],\\
d_j&=\gcd(m_j,q-1),\quad j=1,\dots, n,\qquad D=\text{lcm}[d_1,\dots,d_n],\\
d&=\gcd\Bigl(\,\sum_{j=1}^n\frac {k_jM}{m_j}-kM,\frac{k_1(q-1)}{d_1},\dots,\frac{k_n(q-1)}{d_n},q-1\Bigr).
\end{align*}
Note that $k_0\mid d$. Moreover, if $k_1=\dots=k_n=1$, then
$$
d=\gcd\Bigl(\,\sum_{j=1}^n\frac {k_jM}{m_j}-kM,\frac{q-1}D\,\Bigr).
$$
Carlitz~\cite{C2} studied the case $k_1=\dots=k_n=m_1=\dots=m_n=1$, $k=2$, $p>2$. He proved that in this case
$$
N_q=\begin{cases}
q^2+1&\text{if $n=3$,}\\
q^3-1-\eta(a)q\,\,&\text{if $n=4$.}
\end{cases}
$$
In another paper~\cite{C1}, Carlitz treated the case $m_1=\dots=m_n=2$, $k=k_1=\dots=k_n=1$, $p>2$, $n=3$ or $4$.
 In a series of papers (see~\cite{B1,B2,B3,B4} and references therein) we considered \eqref{eq1} with $k_1=\dots=k_n=1$ and obtained some generalizations of Carlitz's evaluations. Our results cover the following cases:
\begin{itemize}
\item[(a)] 
$m_1=\dots=m_n=1$, $d=1$ or 2 or 3 or 4 or 6 or 8;
\item[(b)] 
$m_1=\dots=m_n=1$, $d=7$ or 14, $p\not\equiv 1\pmod{7}$;
\item[(c)] 
$m_1=\dots=m_n=1$, $k=2$, $d=2^t$, $p\equiv\pm 3\pmod{8}$ or $p\equiv 9\pmod{16}$;
\item[(d)] 
$m_1=\dots=m_n=2$, $k=1$, $p>2$, $d=1$ or 2 or 4;
\item[(e)] 
$m_1=\dots=m_n=2$, $k=1$, $p>2$, $d>1$ and $-1$ is a power of $p$ modulo~$2d$;
\item[(f)] 
$m_1=\dots=m_n=2$, $k=1$, $d=2^t$, $p\equiv\pm 3\pmod{8}$ or $p\equiv 9\pmod{16}$;
\item[(g)] 
$a_1=\dots=a_n=1$, $dD=1$ or 2 or $dD>2$ and $-1$ is a power of $p$ modulo~$dD$.
\end{itemize}

The case of arbitrary $k_1,\dots,k_n$ was recently considered in~\cite{PZC}. By using the augmented degree matrix and Gauss sums, the authors established the following: if 
$$
\gcd\Bigl(\,\sum_{j=1}^n \frac{k_jm_1\cdots m_n}{m_j}-km_1\cdots m_n,q-1\,\Bigr)=1,
$$
then $N_q=q^{n-1}+(-1)^{n-1}$. For a simple combinatorial proof of this latter result, see \cite{B5}.

The goal of this chapter is to find the explicit formulas for $N_q$ in some other cases. Our main results are Theorems~\ref{t1}--\ref{t4} in Sections~\ref{s4} and \ref{s5}. The results of numerical experiments are presented in Section~\ref{s6}.

\section{Preliminary lemmas}

Let $\psi$ be a multiplicative character on $\mathbb
F_q$ (we extend $\psi$ to all of $\mathbb F_q$ by setting
$\psi(0)=1$ if $\psi$ is trivial and $\psi(0)=0$ if $\psi$ is nontrivial).  Define the sum $T(\psi)$ corresponding to the character $\psi$ as
$$
T(\psi)=\frac 1{q-1}\sum_{\substack{x_1,\dots,x_n\in\mathbb
F_q^*\\a_1^{}x_1^{m_1}+\dots+a_n^{}x_n^{m_n}\ne 0}}\psi^{k_1}(x_1)\cdots\psi^{k_n}(x_n) \bar\psi^k(a_1^{}x_1^{m_1}+\dots+a_n^{}x_n^{m_n}).
$$
Let $N_q(0)$ and $N_q^*(0)$ be the number of solutions to the corresponding diagonal equation
\begin{equation}
\label{eq2}
a_1^{}x_1^{m_1}+\dots+a_n^{}x_n^{m_n}=0
\end{equation}
in $\mathbb F_q^n$ and $(\mathbb F_q^*)^n$, respectively. The following lemma relates $N_q$ to $N_q(0)$, $N_q^*(0)$ and $T(\psi)$.

\begin{lemma}
\label{l1}
If $b$ is not a $k_0$\emph{th} power in $\mathbb F_q$, then
$$
N_q^{}=N_q^{}(0)-N_q^*(0).
$$
If $b$ is a $k_0$\emph{th} power in $\mathbb F_q$, then
$$
N_q=k_0(q-1)^{n-1}+N_q(0)-\frac{k_0+q-1}{q-1}\,N_q^*(0)+\sum_{\psi^{k_0}\ne\varepsilon}\psi(b)T(\psi),
$$
where the summation is taken over all multiplicative characters $\psi$ on $\mathbb F_q$ of order not dividing $k_0$.
\end{lemma}

\begin{proof}
Let $N_q^*$ be the number of solutions to \eqref{eq1} in $(\mathbb F_q^*)^n$. Note that the set of solutions to \eqref{eq1} in $\mathbb F_q^n\setminus(\mathbb F_q^*)^n$ is the same as the set of solutions to \eqref{eq2} in $\mathbb F_q^n\setminus(\mathbb F_q^*)^n$. This yields
\begin{equation}
\label{eq3}
N_q=N_q^*+N_q(0)-N_q^*(0).
\end{equation}
If $b$ is not a $k_0$th power in $\mathbb F_q$, it is easy to see that $N_q^*=0$, and so $N_q=N_q(0)-N_q^*(0)$. Suppose that $b$ is a $k_0$th power in $\mathbb F_q$. Let $x_1,\dots,x_n\in\mathbb F_q^*$ with $a_1^{}x_1^{m_1}+\dots+a_n^{}x_n^{m_n}\ne 0$. Then
\begin{align*}
\frac 1{q-1}&\sum_{\psi}\psi(b)\psi^{k_1}(x_1)\cdots \psi^{k_n}(x_n)\bar\psi^k(a_1^{}x_1^{m_1}+\dots+a_n^{}x_n^{m_n})\\
&=\begin{cases}
1&\text{if $(a_1^{}x_1^{m_1}+\dots+a_n^{}x_n^{m_n})^k=bx_1^{k_1}\cdots x_n^{k_n}$,}\\
0&\text{if $(a_1^{}x_1^{m_1}+\dots+a_n^{}x_n^{m_n})^k\ne bx_1^{k_1}\cdots x_n^{k_n}$,}
\end{cases}
\end{align*}
where the summation is taken over all multiplicative characters $\psi$ on $\mathbb F_q$. Hence
\begin{align*}
N_q^*&=\frac 1{q-1}\sum_{\substack{x_1,\dots,x_n\in\mathbb F_q^*\\a_1^{}x_1^{m_1}+\dots+a_n^{}x_n^{m_n}\ne 0}}\sum_{\psi}\psi(b)\psi^{k_1}(x_1)\cdots \psi^{k_n}(x_n)\bar\psi^k(a_1^{}x_1^{m_1}+\dots+a_n^{}x_n^{m_n})\\
&=\frac 1{q-1}\sum_{\psi^{k_0}=\varepsilon}\psi(b)\sum_{\substack{x_1,\dots,x_n\in\mathbb F_q^*\\a_1^{}x_1^{m_1}+\dots+a_n^{}x_n^{m_n}\ne 0}}1+\sum_{\psi^{k_0}\ne\varepsilon}\psi(b)T(\psi)\\
&=k_0\left((q-1)^{n-1}-\frac{N_q^*(0)}{q-1}\right)+\sum_{\psi^{k_0}\ne\varepsilon}\psi(b)T(\psi).
\end{align*}
Substituting this expression into \eqref{eq3}, we deduce the desired result. 
\end{proof}

Our next lemma shows that $T(\psi)$ vanishes when $\psi^d\ne\varepsilon$.

\begin{lemma}
\label{l2}
Let $\psi$ be a multiplicative character of order~$\delta$
on~$\mathbb F_q$. Suppose that $\delta\nmid d$. Then $T(\psi)=0$.
\end{lemma}

\begin{proof}
It is analogous to that of \cite[Lemma~3.2]{B2}.
\end{proof}

Combining Lemmas \ref{l1} and \ref{l2}, we obtain

\begin{corollary}
\label{c1}
If $b$ is a $k_0$\emph{th} power in $\mathbb F_q$, then
$$
N_q=k_0(q-1)^{n-1}+N_q(0)-\frac{k_0+q-1}{q-1}\,N_q^*(0)+\sum_{\substack{\psi^d=\varepsilon\\ \psi^{k_0}\ne\varepsilon}}\psi(b)T(\psi),
$$
where the summation is taken over all multiplicative characters $\psi$ on $\mathbb F_q$ of order dividing $d$ but not $k_0$.
\end{corollary}

\section{Formulas for the number of solutions to diagonal equations}

\begin{lemma}
\label{l3}
Assume that $d_1,\dots,d_t$ are odd, $d_{t+1},\dots,d_n$ are even, $d_1,\dots,d_t$, $d_{t+1}/2,\ldots,$ $d_n/2$ are pairwise coprime, $0\le t\le n$. Then
$$
N_q(0)=q^{n-1}+\begin{cases}
\eta((-1)^{n/2}a_1\cdots a_n)q^{(n-2)/2}(q-1)&\text{if $t=0$ and $n$ is even,}\\
0&\text{otherwise,}
\end{cases}
$$
and
\begin{align*}
N_q^*(0)=\,&\frac{(q-1)^n+(-1)^n(q-1)}q\\
\label{eq6}
&+\begin{cases}
(-1)^n(q-1)\sum\limits_{
j=1}^{[(n-t)/2]}
\eta((-1)^j)\sigma_{2j}(\eta(a_{t+1}),\dots,\eta(a_n))q^{j-1}&\text{if $t<n$},\\
0&\text{if $t=n$},
\end{cases}
\end{align*}
where $\sigma_{2j}(z_1,\dots,z_{n-t})$ are the elementary symmetric polynomials.
\end{lemma}

\begin{proof}
See Theorem~2 in \cite{J}, Theorem~2 in \cite{SY} and the proof of Theorem~3 in \cite{B1}.
\end{proof}

For positive integers $v_1,\dots,v_r$ let
$I(v_1,\dots,v_r)$ denote the number of $r$-tuples\linebreak
$(j_1,\dots,j_r)$ of integers with $1\le j_t\le v_t-1$ $(1\le
t\le r)$ such that $(j_1/v_1)+\dots+(j_r/v_r)$ is an integer.

\begin{lemma}
\label{l4}
Assume that $a_1=\dots=a_n=1$, $D>2$ and there exists a positive integer $\ell$ such that $D\mid(p^{\ell}+1)$, with $\ell$ chosen minimal. Then $2\ell\mid s$,
$$
N_q(0)=q^{n-1}+(-1)^{((s/2\ell)-1)n}q^{(n-2)/2}(q-1)I(d_1,\dots,d_n),
$$
and
\begin{align*}
N_q^*(0)=\,&\frac{(q-1)^n+(-1)^n(q-1)}q\\
&+(-1)^n(q-1)\sum_{r=2}^n(-1)^{rs/2\ell}q^{(r-2)/2}\sum_{1\le j_1<\dots<j_r\le n}
I(d_{j_1},\ldots,d_{j_r}).
\end{align*}
\end{lemma}

\begin{proof}
See Lemma 4.1 and the proof of Theorem 1.1 in \cite{B2}.
\end{proof}

It is known (see \cite[Proposition~6.17]{S}) that
$$
I(\,\underbrace{v,\dots,v}_r\,)=\frac{(v-1)^r+(-1)^r(v-1)}v.
$$
Therefore we have the following corollary (the first part of this corollary is due to Wolfmann~\cite[Corollary~4]{W}).

\begin{corollary}
\label{c2}
Assume that $d_1=\dots=d_n=D$. Under the conditions of Lemma~$\ref{l4}$, we have
$$
N_q(0)=q^{n-1}+(-1)^{((s/2\ell)-1)n}q^{(n-2)/2}(q-1)\cdot\frac{(D-1)^n+(-1)^n(D-1)}D
$$
and
\begin{align*}
N_q^*(0)=\,&\frac{(q-1)^n+(-1)^n(q-1)}q\\
&+(-1)^n(q-1)\sum_{r=2}^n(-1)^{rs/2\ell}q^{(r-2)/2}\binom nr \cdot\frac{(D-1)^r+(-1)^r(D-1)}D.
\end{align*}
\end{corollary}

In the general case
\begin{align*}
I(v_1,\dots,v_r)&=(-1)^r+\sum_{t=1}^r (-1)^{r-t}\sum_{1\le j_1<\dots<j_t\le r}
\frac{v_{j_1}\cdots v_{j_t}}{\text{\rm lcm}[v_{j_1},\dots,v_{j_t}]}\\
&=\frac{(-1)^r}{v_1\cdots v_r}\sum_{t=0}^{v_1\cdots v_r-1}\prod_{\substack{j=1\\v_j\mid t}}^r(1-v_j)
\end{align*}
(see~\cite[Equation~(6.12)]{LN}, \cite[Theorem~1]{SWM}, \cite[Corollary~2.2]{SZW}, and \cite[Theorem~6.18]{S}, respectively).

\section{The case when $b$ is not a $k_0$th power in $\mathbb F_q$}
\label{s4}

Combining Lemmas~\ref{l1}, \ref{l3}, \ref{l4} and Corollary~\ref{c2}, we obtain the following theorems.

\begin{theorem}
\label{t1}
Assume that $d_1,\dots,d_t$ are odd, $d_{t+1},\dots,d_n$ are even, $d_1,\dots,d_t$, $d_{t+1}/2,\ldots,$ $d_n/2$ are pairwise coprime, $0\le t\le n$, and $b$ is not a $k_0$\emph{th} power in $\mathbb F_q$. If $t=0$ and $n$ is even, then
$$
N_q=q^{n-1}-\frac{(q-1)^n+q-1}q-(q-1)\sum_{j=1}^{(n-2)/2}
\eta((-1)^j)\sigma_{2j}(\eta(a_1),\dots,\eta(a_n))q^{j-1};
$$
if $t=n$, then
$$
N_q=q^{n-1}-\frac{(q-1)^n+(-1)^n(q-1)}q;
$$
otherwise
\begin{align*}
N_q=\,&q^{n-1}-\frac{(q-1)^n+(-1)^n(q-1)}q\\
&-(-1)^n (q-1)\sum_{
j=1}^{[(n-t)/2]}
\eta((-1)^j)\sigma_{2j}(\eta(a_{t+1}),\dots,\eta(a_n))q^{j-1}.
\end{align*}
\end{theorem}

\begin{theorem}
\label{t2}
Assume that $a_1=\dots=a_n=1$, $b$ is not a $k_0$\emph{th} power in $\mathbb F_q$, $D>2$ and there exists a positive integer $\ell$ such that $D\mid(p^{\ell}+1)$, with $\ell$ chosen minimal. Then $2\ell\mid s$ and
\begin{align*}
N_q=\,&q^{n-1}-\frac{(q-1)^n+(-1)^n(q-1)}q\\
&-(-1)^n(q-1)\sum_{r=2}^{n-1}(-1)^{rs/2\ell}q^{(r-2)/2}\sum_{1\le j_1<\dots<j_r\le n}
I(d_{j_1},\ldots,d_{j_r}).
\end{align*}
In particular, if $d_1=\dots=d_n=D$, then
\begin{align*}
N_q=\,&q^{n-1}-\frac{(q-1)^n+(-1)^n(q-1)}q\\
&-(-1)^n(q-1)\sum_{r=2}^{n-1}(-1)^{rs/2\ell}q^{(r-2)/2}\binom nr \cdot\frac{(D-1)^r+(-1)^r(D-1)}D.
\end{align*}
\end{theorem}

\section{The case when $b$ is a $k_0$th power in $\mathbb F_q$}
\label{s5}

If $b$ is a $k_0$th power in $\mathbb F_q$, then, in general, one needs to evaluate the sums $T(\psi)$. However, if
\begin{equation}
\label{eq4}
\gcd\Bigl(\,\sum_{j=1}^n\frac {k_jM}{k_0m_j}-\frac{kM}{k_0},\frac{k_1(q-1)}{k_0d_1},\dots,\frac{k_n(q-1)}{k_0d_n},\frac{q-1}{k_0}\Bigr)=1,
\end{equation}
that is, if $k_0=d$, then Corollary~\ref{c1} yields
$$
N_q=k_0(q-1)^{n-1}+N_q(0)-\frac{k_0+q-1}{q-1}\,N_q^*(0).
$$
Combining the latter equality with Lemmas~\ref{l3} and \ref{l4} and Corollary~\ref{c2}, we deduce the following theorems.

\begin{theorem}
\label{t3}
Assume that $d_1,\dots,d_t$ are odd, $d_{t+1},\dots,d_n$ are even, $d_1,\dots,d_t$, $d_{t+1}/2,\ldots,$ $d_n/2$ are pairwise coprime, $0\le t\le n$, and $b$ is a $k_0$\emph{th} power in $\mathbb F_q$. Assume in addition that \eqref{eq4} holds. If $t=0$ and $n$ is even, then
\begin{align*}
N_q=\,&q^{n-1}-1+\frac{(k_0-1)\bigl((q-1)^n-1\bigr)}q-k_0\eta((-1)^{n/2}a_1\cdots a_n)q^{(n-2)/2}\\
&-(k_0+q-1)\sum_{j=1}^{(n-2)/2}
\eta((-1)^j)\sigma_{2j}(\eta(a_1),\dots,\eta(a_n))q^{j-1};
\end{align*}
if $t=n$, then
$$
N_q=q^{n-1}-(-1)^n+\frac{(k_0-1)\bigl((q-1)^n-(-1)^n\bigr)}q;
$$
otherwise
\begin{align*}
N_q=\,&q^{n-1}-(-1)^n+\frac{(k_0-1)\bigl((q-1)^n-(-1)^n\bigr)}q\\
&-(-1)^n(k_0+q-1)\sum_{
j=1}^{[(n-t)/2]}
\eta((-1)^j)\sigma_{2j}(\eta(a_{t+1}),\dots,\eta(a_n))q^{j-1}.
\end{align*}
\end{theorem}

\begin{theorem}
\label{t4}
Assume that $a_1=\dots=a_n=1$, $b$ is a $k_0$\emph{th} power in $\mathbb F_q$, $D>2$ and there exists a positive integer $\ell$ such that $D\mid(p^{\ell}+1)$, with $\ell$ chosen minimal.  Assume in addition that \eqref{eq4} holds. Then $2\ell\mid s$ and
\begin{align*}
N_q=\,&q^{n-1}-(-1)^n+\frac{(k_0-1)\bigl((q-1)^n-(-1)^n\bigr)}q\\
&-(-1)^{((s/2\ell)-1)n}k_0 q^{(n-2)/2}I(d_1,\dots,d_n)\\
&-(-1)^n(k_0+q-1)\sum_{r=2}^{n-1}(-1)^{rs/2\ell}q^{(r-2)/2}\sum_{1\le j_1<\dots<j_r\le n}
I(d_{j_1},\ldots,d_{j_r}).
\end{align*}
In particular, if $d_1=\dots=d_n=D$, then
\begin{align*}
N_q=\,&q^{n-1}-(-1)^n+\frac{(k_0-1)\bigl((q-1)^n-(-1)^n\bigr)}q\\
&-(-1)^{((s/2\ell)-1)n}k_0 q^{(n-2)/2}\cdot\frac{(D-1)^n+(-1)^n(D-1)}D\\
&-(-1)^n(k_0+q-1)\sum_{r=2}^{n-1}(-1)^{rs/2\ell}q^{(r-2)/2}\binom nr \cdot\frac{(D-1)^r+(-1)^r(D-1)}D.
\end{align*}
\end{theorem}

\section{Numerical results}
\label{s6}

Our theoretical results are supported by numerical experiments. Some
numerical results are listed in Tables~\ref{tab1} and \ref{tab2}.

\begin{table}[h]
\caption{The case when $b$ is not a $k_0$th power in $\mathbb F_q$}
\label{tab1}
\newcolumntype{C}{>{\centering\arraybackslash}X}
\begin{tabularx}{\textwidth}{ccCCCccc}
\hline
$q$&$n$&$(a_1,\dots,a_n)$&$(m_1,\dots,m_n)$&$(k_1,\dots,k_n)$&$k$&$k_0$&$N_q$\\
\hline
16 & 5 & $(1,1,1,1,1)$ & $(2,4,6,8,10)$ & $(5,5,10,10,10)$ & 10 & 5 & 18076\\
17 & 6 & $(1,1,1,1,3,5)$ & $(2,6,6,8,10,14)$ & $(4,4,8,8,8,12)$ & 8 & 4 & 433249\\
19 & 6 & $(1,1,1,2,2,2)$ & $(2,2,2,6,14,14)$ & $(3,3,3,3,6,9)$ & 6 & 3 & 684901\\
25 & 5 & $(1,1,1,1,1,1)$ & $(3,9,10,15,18)$ & $(4,4,8,12,16)$ & 8 & 4 & 81553\\
31 & 4 & $(1,1,5,7)$ & $(5,7,9,11)$ & $(2,4,6,8)$ & 10 & 2 & 3661\\
43 & 5 & $(1,1,2,2,3)$ & $(5,8,8,12,28)$ & $(7,7,14,14,28)$ & 21 & 7 & 377665\\
81 & 4 & $(1,1,1,1)$ & $(4,4,12,28)$ & $(8,8,16,32)$ & 24 & 8 & 7041\\
\hline
\end{tabularx}
\end{table}

\begin{table}[h]
\caption{The case when $b$ is a $k_0$th power in $\mathbb F_q$}
\label{tab2}
\newcolumntype{C}{>{\centering\arraybackslash}X}
\begin{tabularx}{\textwidth}{ccCCCccc}
\hline
$q$&$n$&$(a_1,\dots,a_n)$&$(m_1,\dots,m_n)$&$(k_1,\dots,k_n)$&$k$&$k_0$&$N_q$\\
\hline
37 & 6 & $(1,1,1,2,2,2)$ & $(1,2,2,2,4,6)$ & $(9,9,9,9,9,18)$ & 9 & 9 & 539998021\\
47 & 5 & $(1,1,1,5,5)$ & $(3,7,8,12,14)$ & $(2,4,6,8,10)$ & 4 & 2 & 9261921\\
61 & 4 & $(1,1,1,2)$ & $(6,8,10,14)$ & $(6,6,6,6)$ & 12 & 6 & 1289641\\
64 & 4 & $(1,1,1,1)$ & $(9,18,27,36)$ & $(3,3,3,3)$ & 12 & 3 & 781975\\
71 & 3 & $(1,1,1)$ & $(3,15,49)$ & $(7,21,28)$ & 35 & 7 & 34028\\
81 & 5 & $(1,1,1,1,1)$ & $(3,15,30,35,70)$ & $(5,5,15,15,20)$ & 25 & 5 & 205707971\\
97 & 3 & $(1,7,7)$ & $(2,10,12)$ & $(4,4,4)$ & 8 & 4 & 36673\\
\hline
\end{tabularx}
\end{table}

\section{Conclusion}
In a similar manner, making use of our recent results on diagonal equations of the form $x_1^{2^m}+\dots+x_n^{2^m}=0$  \cite{B6}, we are able to determine $N_q$ explicitly when $d_1=\dots=d_n=2^m$ and $p\equiv\pm 3\pmod{8}$. Although this is straightforward in principle, the resulting formulas are quite cumbersome and for this reason are omitted here.


\end{document}